\documentclass[11pt]{amsart}
\usepackage{mathrsfs}

\usepackage{hyperref}

\newcommand{\bpartial}{\bar{\partial}}

\renewcommand{\a}{\alpha}
\newcommand{\p}{\partial}

\newtheorem{theorem}{Theorem}[section]
\newtheorem{lemma}[theorem]{Lemma}
\newtheorem{proposition}[theorem]{Proposition}
\newtheorem{corollary}[theorem]{Corollary}

 \theoremstyle{definition}

\theoremstyle{remark}
\newtheorem{remark}[theorem]{Remark}

\numberwithin{equation}{section}

\begin{document}
\setlength{\baselineskip}{1.2\baselineskip}

\title[Generalized Complex Monge-Amp\`ere Type Equations]
{Generalized complex Monge-Amp\`ere type equations\\
    on closed Hermitian manifolds}

\author{Wei Sun}

\address{Shanghai Center for Mathematical Sciences, Fudan University}
\email{sunwei\_math@fudan.edu.cn}

\begin{abstract}
We study generalized complex Monge-Amp\`ere type equations on closed Hermitian manifolds. We derive {\em a priori} estimates and then prove the existence of solutions. Moreover, the gradient estimate is improved.

%


\end{abstract}

\maketitle

\section{Introduction}
\label{introduction}
\setcounter{equation}{0}
\medskip

Let $(M, \omega)$ be a compact Hermitian manifold of complex dimension $n\geq 2$ and $\chi$ another Hermitian metric on $M$. In local coordinate charts, we shall write
\(
	\omega = \sqrt{-1} \sum_{i,j} g_{i\bar j} dz^i \wedge d\bar z^j
\)
and
\(
	\chi = \sqrt{-1} \sum_{i,j} \chi_{i\bar j} dz^i \wedge d\bar z^j.
\)
For a real $C^2 (M)$ function $u$, we will use the notation $\chi_u = \chi + \sqrt{-1} \p\bpartial u$.

For a smooth positive real function $\psi$ on $M$, we are concerned with the generalized complex Monge-Amp\`ere type equation
\begin{equation}
\label{int-gcmate-equation}
	\left\{
	\begin{aligned}
		\chi^n_u &= \psi \sum^n_{\a = 1} c_\a \chi^{n - \a}_u \wedge \omega^\a , \\
		\chi_u &> 0, \qquad \sup_M u = 0,
	\end{aligned}
	\right.
\end{equation}
where $c_\a$'s are  nonnegative real constants, and $\sum^n_{\a = 1} c_\a > 0$. Following \cite{SW08, FLM11, GSun12}, we set
$
	[\chi] = \{ \chi_u : u \in C^2 (M) \} 
$,
and assume that there is $\chi' \in [\chi]$ satisfying $\chi' > 0$ and
\begin{equation}
\label{int-gcmate-cone-condition}
	n \chi'^{n - 1} > \psi \sum^{n - 1}_{\a = 1} c_\a (n - \a) \chi'^{n - \a - 1} \wedge \omega^\a.
\end{equation}
By an appropriate approximation as in~\cite{Sun2013e}, a smooth $\chi'$ is available. Without loss of generality, we may assume $\chi = \chi'$ throughout this paper. If $\chi$ and $\omega$ are both K\"ahler and $\psi$ is a constant, $\psi$ is uniquely determined by
\begin{equation}
\label{int-kahler-constant}
	\psi = c: = \frac{\int_M \chi^n}{\sum^n_{\a = 1} c_\a \int_M \chi^{n - \a} \wedge \omega^\a} .
\end{equation}

The generalized complex Monge-Amp\`ere type equation is an extension of  Donaldson's equation~\cite{Donaldson99a}. In fact, Donaldson's equation corresponds to the case that $c_2 = \cdots = c_n = 0$ and $\psi$ is constant,
\begin{equation}
\label{int-Donaldson-equation}
	\chi^n_u = \psi \chi^{n - 1}_u \wedge \omega.
\end{equation}
Chen~\cite{Chen00b} also found the equation when he studied the lower bound of Mabuchi energy. The equation is studied by Chen~\cite{Chen00b, Chen04}, Weinkove~\cite{Weinkove04, Weinkove06}, Song and Weinkove~\cite{SW08} by $J$-flow. Their results were extended to the complex Monge-Amp\`ere type equations by Fang, Lai and Ma~\cite{FLM11} also using parabolic flow method. More general cases were treated by Guan and the author~\cite{GSun12}, the author~\cite{Sun2013e, Sun2014e} on Hermitian manifolds using continuity method. Later, the author~\cite{Sun2013p} reproduced the results in \cite{Sun2013e} by parabolic flow method. In these works, the cone condition analogous to condition~\eqref{int-gcmate-cone-condition} is sufficient and necessary. Moreover, a similar equation was studied by Pingali~\cite{Pingali14} on a flat complex 3-torus.

Fang, Lai and Ma~\cite{FLM11} stated a conjecture for the solvability of \eqref{int-gcmate-equation} under condition~\eqref{int-gcmate-cone-condition}. Admitting the famous work of Yau~\cite{Yau78}, Collins and Sz\'ekelyhidi~\cite{CollinsSzekelyhidi2014a} proved the conjecture by continuity method starting from the complex Monge-Amp\`ere equation.
For Donaldson's equation on toric manifolds, the result was used to verify a numerical criterion for the existence of a cone condition, which was proposed by  Lejmi and Sz\'ekelyhidi~\cite{LejmiSzekelyhidi13}.
In this paper, we shall adopt piecewise continuity method due to the author~\cite{Sun2013e} to prove general results without using the solvability results by Yau~\cite{Yau78}, Cherrier~\cite{Cherrier87}, Tosatti and Weinkove~\cite{TWv10a,TWv10b}.

In order to prove the solvability of generalized complex Monge-Amp\`ere type equations, we shall first obtain {\em a priori} estimates.
\begin{theorem}
\label{theorem-estimate}
Let $(M,\omega)$ be a closed Hermitian manifold of complex dimension $n$ and $u$ be a smooth solution to the equation~\eqref{int-gcmate-equation}. Suppose that
\begin{equation}
\label{int-gcmate-cone-condition-1}
	n \chi^{n - 1} > \psi \sum^{n - 1}_{\a = 1} c_\a (n - \a) \chi^{n - \a - 1} \wedge \omega^\a .
\end{equation}
Then there are uniform $C^\infty$ {\em a priori} estimates of $u$.
\end{theorem}

In this paper, we shall first prove the uniform estimate and partial second order estimates. The gradient estimate follows from these estimates and elliptic estimates for the Laplacian, while it can also be obtained by combining the interior gradient estimate  and the uniform estimate. Adapting the approach of Phong and Sturm~\cite{PhongSturm10, PhongSturm12} and Blocki~\cite{Blocki09a}, we shall improve the gradient estimate due to Guan and the author~\cite{GSun12}. Higher order estimates can be achieved by Evans-Krylov theory~\cite{Evans82, Krylov82,TWWvY2014} and Schauder estimate.

For general Hermitian manifolds, we have the following result.
\begin{theorem}
\label{theorem-hermitian}
Let $(M,\omega)$ be a closed Hermitian manifold of complex dimension $n$ and $\chi$ also a Hermitian metric. Suppose that
\begin{equation*}
	n \chi^{n - 1} > \psi \sum^{n - 1}_{\a = 1} c_\a (n - \a) \chi^{n - \a - 1} \wedge \omega^\a .
\end{equation*}
If there is a $C^2$ function $v$ satisfying $\chi_v > 0$ and
\begin{equation*}
	\chi^n_v \leq \psi \sum^n_{\a = 1} c_\a \chi^{n - \a}_v \wedge \omega^\a,
\end{equation*}
then there exist a unique solution $u$ and a unique constant $b$ such that
\begin{equation}
\label{int-gcmate-equation-actual}
	\left\{
	\begin{aligned}
		\chi^n_u &= e^b \psi \sum^n_{\a = 1} c_\a \chi^{n - \a}_u \wedge \omega^\a , \\
		\chi_u &> 0, \qquad \sup_M u = 0 .
	\end{aligned}
	\right.
\end{equation}

\end{theorem}

A corollary immediately follows from Theorem~\ref{theorem-hermitian} and the fact that the inequality~\eqref{int-gcmate-cone-condition} does not contain the $\omega^n$ term.
\begin{corollary}
\label{corollary-hermitian}
Let $(M,\omega)$ be a closed Hermitian manifold of complex dimension $n$ and $\chi$ also a Hermitian metric. Suppose that
\begin{equation*}
	n \chi^{n - 1} > \psi \sum^{n - 1}_{\a = 1} c_\a (n - \a) \chi^{n - \a - 1} \wedge \omega^\a .
\end{equation*}
Then there is a constant $K\geq 0$ such that if $c_n \geq K$ there exists a unique solution $u$ and a unique constant satisfying the equation~\eqref{int-gcmate-equation-actual}.
\end{corollary}
\begin{remark}
Corollary~\ref{corollary-hermitian} indeed includes the case of complex Monge-Amp\`ere equation, which has a stronger cone condition than all others.
\end{remark}

For K\"ahler manifolds, we have a stronger result.
\begin{theorem}
\label{theorem-kahler}
Let $(M,\omega)$ be a closed K\"ahler manifold of complex dimension $n$ and $\chi$ also K\"ahler. Suppose that
\begin{equation*}
	n \chi^{n - 1} > \psi \sum^{n - 1}_{\a = 1} c_\a (n - \a) \chi^{n - \a - 1} \wedge \omega^\a .
\end{equation*}
If $\psi \geq c$ for all $x \in M$ where $c$ is defined in \eqref{int-kahler-constant}, then there exist a unique solution $u$ and a unique constant $b$  satisfying the equation~\eqref{int-gcmate-equation-actual}.

\end{theorem}


\bigskip

\section{Preliminary}
\label{preliminary}
\setcounter{equation}{0}
\medskip

We denote by $\nabla$ the Chern connection of $g$.
As in \cite{GSun12, Sun2013e}, we express
\begin{equation}
\label{int:definition-X}
X := \chi_u \,,
\end{equation}
and thus
\begin{equation}
\label{int:definition-X-coefficients}
X_{i\bar j} = \chi_{i\bar j} +  \bpartial_j \p_i u\,.
\end{equation}
Also, we denote the coefficients of $X^{-1}$ by $X^{i\bar j}$.
Recall that $S_\a (\lambda)$ denote the $\a$-th elementary symmetric polynomial of $\lambda \in \mathbb{R}^n$,
\begin{equation}
	S_\a (\lambda) = \sum_{1 \leq i_1 < \cdots < i_\a \leq n} \lambda_{i_1} \cdots \lambda_{i_\a} \,.
\end{equation}
Also, $\lambda_* (X)$ is the eigenvalue set of $X$ with respect to $\{g_{i\bar j}\}$ while $\lambda^* (X^{-1})$ is the eigenvalue set of $X^{-1}$ with respect to $\{g^{i\bar j}\}$.
Thus we write $S_\a (X) = S_\a (\lambda_* (X))$ and $S_\a (X^{-1}) = S_\a (\lambda^* (X^{-1}))$. 
For convenience, we shall use $S_\alpha$ to denote $S_\alpha(X^{-1})$. In local coordinates, equation~\eqref{int-gcmate-equation} can be written in the form
\begin{equation}
\label{int-gcmate-equation-equivalent}
    	F(\chi_u) := - \sum^n_{\a = 1} \frac{c_\a}{C^\a_n} S_\a =  - \frac{1}{\psi} ,
\end{equation}
where $C^\a_n  = \frac{n!}{(n - \a)! \a !}$. For $\{i_1,\cdots,i_s\} \subseteq \{1,\cdots,n\}$,
\begin{equation}
 	S_{k;i_1\cdots i_s} (\lambda) = S_k (\lambda|_{\lambda_{i_1} = \cdots = \lambda_{i_s}= 0}).
\end{equation}
By convention, $S_{0;k} = 1$. Then inequality~\eqref{int-gcmate-cone-condition} is equivalent to
\begin{equation}
  	\frac{1}{\psi} > \sum^n_{\a = 1} \frac{c_\alpha}{C^\a_n} S_{\alpha }((\chi|k)^{-1})
\end{equation}
for all $k$, where $(\chi|k)$ is the $(k,k)$-minor matrix of $\chi$ in local charts. Moreover, we may assume
\begin{equation}
    \epsilon\omega \leq \chi \leq \epsilon^{-1} \omega
\end{equation}
for some $\epsilon > 0$.

Define $F^{i\bar j} := \frac{\partial F}{\partial u_{i\bar j} }$. Assume that at the point $p$,  $g_{i\bar j} = \delta_{ij}$ and $X_{i\bar j}$ is diagonal in a specific chart. Thus $F^{i\bar j}$ is also diagonal at $p$ and
\begin{equation}
	F^{i\bar i} = \sum^n_{\alpha = 1} \frac{c_\alpha}{C^\alpha_n} S_{\alpha - 1;i} (X^{i\bar i})^2 .
\end{equation}
By direct calculation,
\begin{equation}
	F^{i\bar i} X_{i\bar i} = \sum^n_{\alpha = 1} \frac{c_\alpha}{C^\alpha_n} S_{\alpha - 1;i} X^{i\bar i} \leq \sum^n_{\alpha = 1} \frac{c_\alpha}{C^\alpha_n} S_{\alpha } = \frac{1}{\psi} ,
\end{equation}
and
\begin{equation}
	\sum_i F^{i\bar i} X_{i\bar i} = \sum^n_{\alpha = 1} \frac{c_\alpha}{C^\alpha_n} \sum_i S_{\alpha - 1;i} X^{i\bar i} = \sum^n_{\alpha = 1} \frac{\alpha c_\alpha}{C^\alpha_n} S_\alpha \in \left[\frac{1}{\psi} , \frac{n}{\psi}\right] .
\end{equation}
Also
\begin{equation}
\begin{aligned}
	\sum_i F^{i\bar i} &= \sum^n_{\alpha = 1} \frac{c_\alpha}{C^\alpha_n} \sum_i S_{\alpha - 1;i} (X^{i\bar i} )^2 \\
	&= \sum^{n - 1}_{\alpha = 1} \frac{c_\alpha}{C^\alpha_n} \Big(S_\alpha S_1 - (\alpha + 1) S_{\alpha + 1} \Big) + c_n S_n S_1 ,
\end{aligned}
\end{equation}
and by generalized Newton-Maclaurin inequalities,
\begin{equation}
	\sum_i F^{i\bar i} 
	= S_1 \sum^n_1 \frac{\alpha c_\alpha}{n C^\alpha_n} S_\alpha
	\geq \frac{S_1}{n \psi} .
\end{equation}

As in \cite{FLM11, GSun12, Sun2013e}, differentiating $S_\a$ twice and applying the strong concavity of $S_\a$ \cite{GLZ}, we have at point $p$,
\begin{equation}
\label{int-formula-S-partial}
\begin{aligned}
	\partial_l (S_\a)  = \; -\sum_i S_{\alpha-1;i} (X^{i\bar i})^2 X_{i\bar il}
\end{aligned}
\end{equation}
and
\begin{equation}
\label{int-formula-S-double-derivative}
\begin{aligned}
	\bar\partial_l\partial_l (S_\a) \geq   \sum_{i,j} S_{\alpha -1;i}  (X^{i\bar i})^2 X^{j\bar j}X_{j\bar i\bar l}X_{i\bar j l} - \sum_i S_{\alpha -1;i}(X^{i\bar i})^2 X_{i\bar il\bar l} .
\end{aligned}
\end{equation}
Therefore
\begin{equation}
\label{formula-equation-partial}
	- \p_l \Big(\frac{1}{\psi}\Big) = \sum_i F^{i\bar i} X_{i\bar il}
\end{equation}
and
\begin{equation}
\label{formula-equation-double-derivative}
	\bpartial_l\p_l \Big(\frac{1}{\psi}\Big) \geq \sum_{i,j} F^{i\bar i} X^{j\bar j} X_{j\bar i\bar l} X_{i\bar j l} - \sum_i F^{i\bar i} X_{i\bar il\bar l} .
\end{equation}

\bigskip

\section{The uniform estimate}
\label{uniform}
\setcounter{equation}{0}
\medskip

In this section, we derive the uniform estimate directly by Moser iteration. However, the $C^2$ estimate in Section~\ref{C2} also implies a uniform estimate as shown in \cite{TWv10a}. Moreover, the uniform estimate also follows from Proposition 10  in \cite{Szekelyhidi2014b}.

\begin{theorem}
\label{gcmate-theorem-uniform}
Let $u$ be a smooth solution to the equation
\begin{equation}
\label{generalized-quotient-equation}
	\left\{
	\begin{aligned}
		&\chi^m_u \wedge \omega^{n - m} = \psi \sum^m_{\a = 1} c_\a \chi^{m - \a}_u \wedge \omega^{n - m + \a}, \\
		&\chi_u \in \Gamma^m_\omega \cap [\chi], \qquad \sup_M u = 0,
	\end{aligned}
	\right.
\end{equation}
where $1 \leq m \leq n$, $\chi \in \Gamma^m_\omega$ is a smooth real $(1,1)$ form, $c_\a$'s are nonnegative real constants and $\sum^m_{\a = 1} c_\a > 0$,  where $\Gamma^m_\omega$ is the set of all the real $(1, 1)$ forms whose eigenvalue set with respect to $\omega$ belong to $m$-positive cone in $\mathbb{R}^n$. If
\begin{equation}
\label{uniform-kahler-cone-condition}
	m \chi^{m - 1} \wedge \omega^{n - m} > \psi \sum^{m - 1}_{\a = 1} c_\a (m - \a) \chi^{m - \a - 1} \wedge \omega^{n - m + \a},
\end{equation}
there is a constant $C$ such that $\sup_M |u| < C$.
\end{theorem}

Following the work of Tosatti and Weinkove~\cite{TWv10a, TWv10b}, it suffices to show that there is a constant $C$ such that
\begin{equation}
\label{uniform-key-inequality}
	\int_M |\p e^{- \frac{p}{2} u}|^2 \omega^n \leq C p \int_M e^{- p u} \omega^n
\end{equation}
for $p$ large enough. We refer the readers to \cite{Yau78, TWv10a, TWv10b, TWv13b} for more details.

Applying the technique in \cite{Sun2014e,Sun2014u}, we obtain the following lemma on closed Hermitian manifolds.
\begin{lemma}
\label{uniform-lemma-hermitian}
Under the assumptions of Theorem~\ref{gcmate-theorem-uniform},
inequality~\eqref{uniform-key-inequality} holds true for some uniform constants $C$, $p_0$ such that for all $p \geq p_0$.
\end{lemma}

We recall the following two lemmas proven in \cite{Sun2014u}. 
\begin{lemma}
\label{gcmate-lemma-iteration}
Suppose that there are constants $\Lambda \geq \lambda > 0$ such that $\chi - \lambda \omega, \Lambda \omega - \chi \in \Gamma^m_\omega$. For $2 \leq k \leq m$, we have
\begin{equation}
\begin{aligned}
	&  \int^1_0 \left(\int_M e^{- p u} \sqrt{- 1} \p u \wedge \bpartial u \wedge \chi^{k - 1}_{t u} \wedge \omega^{n - k} \right) dt \\
	&\qquad\geq \frac{(k - 1)\lambda}{k} \int^1_0 \left(\int_M e^{- p u} \sqrt{- 1} \p u \wedge \bpartial u \wedge \chi^{k - 2}_{t u} \wedge \omega^{n - k + 1} \right) dt . 
\end{aligned}
\end{equation}
and
\begin{equation}
	 \int^1_0 \left(\int_M e^{- p u} \chi^k_{t u} \wedge \omega^{n - k}\right) dt \geq \frac{k\lambda}{k + 1} \int^1_0 \left(\int_M e^{- p u} \chi^{k - 1}_{t u} \wedge \omega^{n - k + 1}\right) dt  .
\end{equation}
\end{lemma}

The second lemma follows from a lemma of Zhang~\cite{Zhang2015} and Lemma~\ref{gcmate-lemma-iteration}.
\begin{lemma}
\label{gcmate-lemma-1}
There is a uniform constant $C > 0$ such that
\begin{equation}
\label{gcmate-lemma-1-inequality}
\begin{aligned}
	&\, \left|\int^1_0 \left(\int_M e^{- p u} \chi^{m - 1}_{t u} \wedge \sqrt{- 1} \p \bpartial \omega^{n - m} \right) dt \right| \\
	\leq&\, C p \int^1_0 \left(\int_M e^{- p u} \sqrt{- 1} \p u \wedge \bpartial u \wedge \chi^{m - 2}_{t u} \wedge \omega^{n - m + 1}\right) dt \\
	&\,  + C \int^1_0 \left(\int_M e^{- p u} \chi^{m - 2}_{t u} \wedge \omega^{n - m + 2}\right) dt  .
\end{aligned}
\end{equation}
\end{lemma}

\begin{proof}[Proof of Lemma~\ref{uniform-lemma-hermitian}]

Consider the integral
\begin{equation}
\label{uniform-lemma-kahler-integral}
\begin{aligned}
	I &= \int_M e^{- p u} \Big[(\chi^m_u  - \chi^m ) \wedge \omega^{n - m} - \psi \sum^{m - 1}_{\a = 1}  c_\a(\chi^{m - \a}_u - \chi^{m - \a} )\wedge \omega^{n - m + \a}\Big] \\
	&= \int_M \int^1_0 e^{- p u} \sqrt{- 1} \p\bpartial u \wedge \Big[ m \chi^{m - 1}_{t u} \wedge \omega^{n - m}  \\
	&\hspace{12em} - \psi \sum^{m - 1}_{\a = 1} c_\a (m - \a) \chi^{m - \a - 1}_{t u} \wedge \omega^{n - m + \a} \Big] dt.
\end{aligned}
\end{equation}
It is easy to see that for some constant $C$,
\begin{equation}
\label{uniform-lemma-hermitian-integral-less}
	I \leq C \int_M e^{- p u} \chi^n .
\end{equation}

On the other hand, using integration by parts,
\begin{equation}
\label{uniform-lemma-hermitian-integral-greater}
\begin{aligned}
	I =&\, p \int^1_0 \int_M e^{- p u} \sqrt{- 1} \p u \wedge \bpartial u \wedge \Big[ m \chi^{m - 1}_{t u} \wedge \omega^{n - m} \\
	&\,\,\,\hspace{12em}- \psi \sum^{m - 1}_{\a = 1} c_\a (m - \a) \chi^{m - \a - 1}_{t u} \wedge \omega^{n - m + \a} \Big] dt \\
	&\, - \frac{m}{p} \int^1_0 \int_M e^{- p u} \chi^{m - 1}_{t u} \wedge \sqrt{- 1} \p\bpartial \omega^{n - m} dt \\
	&\, + \frac{m}{p} \int^1_0 \int_M e^{- p u} \sqrt{ - 1} \p\bpartial \chi^{m - 1}_{t u} \wedge \omega^{n - m} dt \\
	&\, + m (m - 1) \int^1_0 \int_M e^{- p u} \sqrt{ - 1} \bpartial u \wedge \p\chi \wedge \chi^{m - 2} \wedge \omega^{n - m} dt \\
	&\, - m (m - 1) \int^1_0 \int_M e^{- p u} \sqrt{ - 1} \p u \wedge \bpartial \chi \wedge \chi^{m - 2} \wedge \omega^{n - m} dt \\
	&\, - \sum^{m - 1}_{\a = 1} \int^1_0 \int_M e^{- p u} \psi \sqrt{- 1} \bpartial u \wedge \p \eta_\a\wedge \chi^{m - \a - 1}_{t u} \wedge \omega^{n - m + \a - 1} dt \\
	&\, - \sum^{m - 1}_{\a = 1} c_\a (m - \a) \int^1_0 \int_M e^{- p u} \sqrt{- 1} \bpartial u \wedge \p \psi \wedge \chi^{m - \a - 1}_{t u} \wedge \omega^{n - m + \a} dt ,
\end{aligned}
\end{equation}
where $c_0 = - 1 $ and 
\begin{equation}
	\eta_\a = c_{\a - 1} (m - \a + 1) (m - \a) \chi + c_\a (m - \a) (n - m + \a) \omega .
\end{equation}
Because of the concavity of hyperbolic functions, the first term in \eqref{uniform-lemma-hermitian-integral-greater} is greater than
\begin{equation}
\label{uniform-lemma-hermitian-integral-greater-positive-term}	
\begin{aligned}
	&\, 2 a_1 p  \int_M e^{- p u} \sqrt{- 1} \p u \wedge \bpartial u \wedge \chi^{m - 1} \wedge \omega^{n- m} \\
	+ &\, 2 a_1 p  \int^{\frac{1}{2}}_0\int_M e^{- p u} \sqrt{- 1} \p u \wedge \bpartial u \wedge \chi^{m - 1}_{t u} \wedge \omega^{n - m} dt ,
\end{aligned}
\end{equation}
for some small $a_1 > 0$. Applying Schwarz's inequality pointwise and Lemma~\ref{gcmate-lemma-iteration}, we can find a lower bound for the last five terms in \eqref{uniform-lemma-hermitian-integral-greater}, for any $\epsilon_1 > 0$,
\begin{equation}
\label{uniform-lemma-hermitian-integral-greater-low-bound}
\begin{aligned}
	& - \epsilon_1 p \int^1_0 \int_M e^{ - p u} \sqrt{- 1} \p u \wedge \bpartial u \wedge \chi^{m - 2}_{t u} \wedge \omega^{n - m + 1} dt \\
	& - \frac{C_1}{\epsilon_1 p} \int^1_0 \int_M e^{- p u} \chi^{m - 2}_{t u} \wedge \omega^{n - m + 2} dt.
\end{aligned}
\end{equation}
Then  we have
\begin{equation}
\label{uniform-lemma-hermitian-integral-greater-0-0}
\begin{aligned}
	I \geq&\, 2 a_1 p  \int_M e^{- p u} \sqrt{- 1} \p u \wedge \bpartial u \wedge \chi^{m - 1} \wedge \omega^{n- m} \\
	&\, + 2 a_1 p  \int^{\frac{1}{2}}_0\int_M e^{- p u} \sqrt{- 1} \p u \wedge \bpartial u \wedge \chi^{m - 1}_{t u} \wedge \omega^{n - m} dt \\
	&\, - \frac{m}{p} \int^1_0 \int_M e^{- p u} \chi^{m - 1}_{t u} \wedge \sqrt{- 1} \p\bpartial \omega^{n - m} dt \\
	&\,- \epsilon_1 p \int^1_0 \int_M e^{ - p u} \sqrt{- 1} \p u \wedge \bpartial u \wedge \chi^{m - 2}_{t u} \wedge \omega^{n - m + 1} dt \\
	&\, - \frac{C_1}{\epsilon_1 p} \int^1_0 \int_M e^{- p u} \chi^{m - 2}_{t u} \wedge \omega^{n - m + 2} dt .
\end{aligned}
\end{equation}

If $m = 2$, then when $\epsilon_1$ is small enough
\begin{equation}
\label{uniform-lemma-hermitian-integral-greater-0}
\begin{aligned}
	I \geq&\,  a_1 p  \int_M e^{- p u} \sqrt{- 1} \p u \wedge \bpartial u \wedge \chi \wedge \omega^{n - 2} \\
	&\, - \frac{2}{p} \int^1_0 \int_M e^{- p u} \chi_{t u} \wedge \sqrt{- 1} \p\bpartial \omega^{n - 2} dt - \frac{C_1}{\epsilon_1 p} \int_M e^{- p u} \omega^{n} .
\end{aligned}
\end{equation}
We only need to control the second term,
\begin{equation}
\begin{aligned}
	&\, - \frac{2}{p} \int^1_0 \int_M e^{- p u} \chi_{t u} \wedge \sqrt{- 1} \p\bpartial \omega^{n - 2} dt \\
	=&\, - \frac{2}{p} \int_M e^{- p u} \chi \wedge \sqrt{- 1} \p\bpartial \omega^{n - 2}  + \int_M e^{- p u} \p u \wedge \bpartial u \wedge \p\bpartial \omega^{n - 2} \\
	\geq&\, - \frac{C_2}{p} \int_M e^{- p u} \omega^n - C_3 \int_M e^{- p u} \sqrt{- 1} \p u \wedge \bpartial u \wedge \omega^{n - 1}.
\end{aligned}
\end{equation}
Choosing $p$ large enough, we obtain
\begin{equation}
	I \geq \frac{a_1 p}{2} \int_M e^{- p u} \sqrt{- 1} \p u \wedge \bpartial u \wedge \chi \wedge \omega^{n - 2} - \frac{C_4}{p} \int_M e^{- p u} \omega^{n} .
\end{equation}

If $m \geq 3$, we use Lemma~\ref{gcmate-lemma-1} to control the most troublesome term,
\begin{equation}
\label{uniform-lemma-hermitian-integral-greater-0-1}
\begin{aligned}
	I \geq&\, 2 a_1 p  \int_M e^{- p u} \sqrt{- 1} \p u \wedge \bpartial u \wedge \chi^{m - 1} \wedge \omega^{n- m} \\
	&\, + 2 a_1 p  \int^{\frac{1}{2}}_0\int_M e^{- p u} \sqrt{- 1} \p u \wedge \bpartial u \wedge \chi^{m - 1}_{t u} \wedge \omega^{n - m} dt \\
	&\,- C_5 \int^1_0 \int_M e^{- p u} \sqrt{- 1} \p u \wedge \bpartial u \wedge \chi^{m - 2}_{t u} \wedge \omega^{n - m + 1} dt \\
	&\,- \epsilon_1 p \int^1_0 \int_M e^{ - p u} \sqrt{- 1} \p u \wedge \bpartial u \wedge \chi^{m - 2}_{t u} \wedge \omega^{n - m + 1} dt \\
	&\, 
	- \frac{ C_6}{p} \int^1_0 \int_M e^{- p u} \chi^{m - 2}_{t u} \wedge \omega^{n - m + 2} dt  .
\end{aligned}
\end{equation}
Note that $C_6$ depends on $\epsilon_1$. As shown in \cite{Sun2014u}, if $p$ is sufficiently large,
\begin{equation}
\begin{aligned}
	&\, \frac{1}{p} \int^1_0 \int_M e^{- p u} \chi^{m - 2}_{t u} \wedge \omega^{n - m + 2} dt \\
	\leq&\, C_7 \int^1_0 \int_M e^{- p u} \sqrt{- 1} \p u \wedge \bpartial u \wedge \chi^{m - 3}_{t u} \wedge \omega^{n - m + 2} dt \\
	&\,\qquad + \frac{2}{p} \int_M e^{- p u} \chi^{m - 2} \wedge \omega^{n - m + 2} .
\end{aligned}
\end{equation}
And thus
\begin{equation}
\label{uniform-lemma-hermitian-integral-greater-0-2}
\begin{aligned}
	I \geq&\, 2 a_1 p  \int_M e^{- p u} \sqrt{- 1} \p u \wedge \bpartial u \wedge \chi^{m - 1} \wedge \omega^{n- m} \\
	&\, + 2 a_1 p  \int^{\frac{1}{2}}_0\int_M e^{- p u} \sqrt{- 1} \p u \wedge \bpartial u \wedge \chi^{m - 1}_{t u} \wedge \omega^{n - m} dt \\
	&\,- C_5 \int^1_0 \int_M e^{- p u} \sqrt{- 1} \p u \wedge \bpartial u \wedge \chi^{m - 2}_{t u} \wedge \omega^{n - m + 1} dt \\
	&\,- \epsilon_1 p \int^1_0 \int_M e^{ - p u} \sqrt{- 1} \p u \wedge \bpartial u \wedge \chi^{m - 2}_{t u} \wedge \omega^{n - m + 1} dt \\
	&\, 
	-  C_8 \int^1_0 \int_M e^{- p u} \sqrt{- 1} \p u \wedge \bpartial u \wedge \chi^{m - 3}_{t u} \wedge \omega^{n - m + 2} dt\qquad \\
	&\, - \frac{2 C_6}{p} \int_M e^{- p u} \chi^{m - 2} \wedge \omega^{n - m + 2} .
\end{aligned}
\end{equation}
Using Lemma~\ref{gcmate-lemma-iteration} and the fact
\begin{equation}
\begin{aligned}
	&\,\int^1_0 \int_M e^{- p u} \sqrt{- 1} \p u \wedge \bpartial u \wedge \chi^{m - 2}_{t u} \wedge \omega^{n - m + 1} dt \\
	\leq&\, 2^{m - 1} \int^{\frac{1}{2}}_0 \int_M e^{- p u} \sqrt{- 1} \p u \wedge \bpartial u \wedge \chi^{m - 2}_{t u} \wedge \omega^{n - m + 1} dt,
\end{aligned}
\end{equation}
we can choose $\epsilon_1$ small enough and then $p$ large enough such that
\begin{equation}
\label{uniform-lemma-hermitian-integral-greater-0-3}
\begin{aligned}
	I \geq&\, 2 a_1 p  \int_M e^{- p u} \sqrt{- 1} \p u \wedge \bpartial u \wedge \chi^{m - 1} \wedge \omega^{n- m} \\
	&\, 
	- \frac{2 C_6}{p} \int_M e^{- p u} \chi^{m - 2} \wedge \omega^{n - m + 2}  .
\end{aligned}
\end{equation}

\end{proof}

\bigskip

\section{The second order estimate}
\label{C2}
\setcounter{equation}{0}
\medskip
In this section, we prove the partial second order estimates. Note that the  sharp form of estimates also implies the uniform estimate as shown in \cite{TWv10a}.

In order to obtain the second order estimate, we need the following lemma. There are more general statements by Guan~\cite{Guan2014a},  Collins and   Sz\'ekelyhidi~\cite{CollinsSzekelyhidi2014a},  Sz\'ekelyhidi~\cite{Szekelyhidi2014b} respectively. However, for completeness we include a proof following \cite{FLM11,CollinsSzekelyhidi2014a}.
\begin{lemma}
\label{lemma-alternative}
There are constants $N$, $\theta > 0$ such that when $w > N$ at a point $p$,
\begin{equation}
\label{lemma-alternative-inequality}
	\sum_i F^{i\bar i} u_{i\bar i} \leq - \theta \Bigg(\sum_i  F^{i\bar i}  + 1\Bigg) ,
\end{equation}
under coordinates around $p$ such that $g_{i\bar j} = \delta_{ij}$ and $X_{i\bar j}$ is diagonal at the point $p$.
\end{lemma}
\begin{proof}
Without loss of generality, we may assume that $X_{1\bar 1} \geq \cdots \geq X_{n\bar n}$. By direct calculation,
\begin{equation}
	\sum_i S_{\a - 1;i} (X^{i\bar i})^2 u_{i\bar i}
	\leq \a S_\a - \epsilon \sum_i S_{\a - 1;i} (X^{i\bar i})^2 \\
	\leq \Big(1 - \frac{\epsilon}{n}S_1\Big) \a S_\a ,
\end{equation}
which means that if the inequality~\eqref{lemma-alternative-inequality} does not hold for any $\theta > 0$,
\begin{equation}
	\sum_i X^{i\bar i} \leq \frac{2 n}{\epsilon}.
\end{equation}
Then we have
\begin{equation}
	X_{1\bar 1} \geq \cdots \geq X_{n\bar n} \geq \frac{\epsilon}{2 n}.
\end{equation}

Now we follow the argument in \cite{CollinsSzekelyhidi2014a}
\begin{equation}
\begin{aligned}
	 \sum_i  F^{i\bar i}  u_{i\bar i} 
	&\leq \sum^n_{\a = 1} \frac{c_\a}{C^\a_n} \sum_i S_{\a - 1;i} (X^{i\bar i})^2 X_{i\bar i} - \sum^n_{\a = 1} \frac{c_\a}{C^\a_n} \sum_i S_{\a - 1;1i} (X^{i\bar i})^2 \chi_{i\bar i} \\
	&\leq  \sum^n_{\a = 1} \frac{(\a - 1) c_\a}{C^\a_n} (S_{\alpha} - S_{\alpha;1}) - \sum^n_{\a = 1} \frac{c_\a}{C^\a_n} (S_{\a ;1} - S_{\a }((\chi|1)^{-1})) .
\end{aligned}
\end{equation}
 Then by the condition~\eqref{int-gcmate-cone-condition-1} there is a constant $\sigma > 0$ such that
\begin{equation}
\begin{aligned}
	\sum^n_{\a = 1} \frac{c_\a}{C^\a_n} \sum_i S_{\a - 1;i} (X^{i\bar i})^2 u_{i\bar i} &\leq
	- \frac{\sigma}{\psi} + \sum^n_{\a = 1} \frac{\alpha c_\a}{C^\a_n} (S_\a - S_{\a;1} ) \\
	&= - \frac{\sigma}{\psi} + \sum^n_{\a = 1} \frac{\alpha c_\a}{C^\a_n} X^{1\bar 1} S_{\a - 1;1} \\
	&\leq
	 - \frac{\sigma}{\psi} + \sum^n_{\a = 1} \frac{\alpha^2 c_\alpha}{n} \left(\frac{2 n}{\epsilon}\right)^{\a - 1} X^{1\bar 1} .
\end{aligned}
\end{equation}
So if $X_{1\bar 1}$ is large enough, we achieve the inequality~\eqref{lemma-alternative-inequality}.

\end{proof}

\begin{proposition}
 Let $u\in C^4(M)$ be a solution to equation \eqref{int-gcmate-equation} and $w = \Delta_\omega u + tr_\omega\chi $. Then there are uniform positive constants $C$ and $A$ such that
\begin{equation}
\label{C2-1}
\sup_{M} w \leq C e^{A (u - \inf_M u)},
\end{equation}
where $C$, $A$ depend on the given geometric quantities.
\end{proposition}

\begin{proof}
The proof follows the argument in \cite{Sun2013e} closely, so we give a sketch here. Let us consider the function $ \ln w + \phi$ where 
\begin{equation}
\label{C2-test-function}
    	\phi := - A u + \frac{1}{u - \inf_M u + 1} = - A u + E_1
\end{equation}
by a trick due to Phong and Sturm \cite{PhongSturm10}.
Without loss of generality, we assume $C, A \gg 1$ throughout this section. We may also assume that $w \gg 1$. 

Suppose that $\ln w + \phi$ attains its maximal value at some point $p\in M$. Choose a local chart near $p$ such that $g_{i\bar j} = \delta_{ij}$ and $X_{i\bar j}$ is diagonal at $p$. 
%
As in \cite{Cherrier87, TWv11, Sun2013e}, direct calculation shows that
\begin{equation}
\label{C2-formula-sum-1}
\begin{aligned}
	&\, \sum_{i,j,l} S_{\a - 1;i} (X^{i\bar i})^2 X^{j\bar j} X_{j\bar i\bar l} X_{i\bar j l} - \sum_{i,l} S_{\a - 1;i} (X^{i\bar i})^2 X_{i\bar il\bar l} \\
	\geq&\, w \sum_i S_{\a - 1;i} (X^{i\bar i})^2 \bpartial_i\p_i \phi - \frac{2}{w} \sum_{i,j} S_{\a - 1;i} (X^{i\bar i})^2 \mathfrak{Re} \{ \sum_k \hat{T}^k_{ij} \chi_{k\bar j}{\bpartial_i w}\} \\
	&\, - \a C_1 S_\a - C_2 w \sum_i S_{\a - 1} (X^{i\bar i})^2
\end{aligned}
\end{equation}
where $\hat{T}$ denotes the torsion with respect to the Hermitian metric $\chi$.
%
%
%
It is easy to see that
\begin{equation}
\label{C2-test-function-first-derivative}
	\p_i \phi = - (A + E^2_1) u_i
\end{equation}
and
\begin{equation}
\label{C2-test-function-second-derivative}
\begin{aligned}
    	\bpartial_i \p_i \phi = - (A + E^2_1) u_{i\bar i} + 2 E^3_1 |u_i|^2.
\end{aligned}
\end{equation}
Then, we have
\begin{equation}
\begin{aligned}
     	 &\quad w \sum_i S_{\a -1;i} (X^{i\bar i})^2 \bpartial_i\p_i \phi \\
	 &= 
     	2 E^3_1 w  \sum_i S_{\a -1;i} (X^{i\bar i})^2 |u_i|^2 - ( A + E^2_1 ) w \sum_i S_{\a -1;i} (X^{i\bar i})^2 u_{i\bar i}  ,
\end{aligned}
\end{equation}
and
\begin{equation}
\begin{aligned}
	&\,  2  \sum_{i,j} S_{\a - 1;i} (X^{i\bar i})^2 \mathfrak{Re}\Big\{\sum_k \hat{T}^k_{ij} \chi_{k\bar j} {\bpartial_i \phi}\Big\} \\
    	\geq&\, - w E^3_1 \sum_i S_{\a -1;i} (X^{i\bar i})^2 |u_i|^2 - \frac{C_3 A^2}{w E^3_1} \sum_i S_{\a -1;i} (X^{i\bar i})^2 .
\end{aligned}
\end{equation}
Therefore,
\begin{equation}
\label{C2-formula-sum-3}
	 C_2 w \sum_i F^{i\bar i}   + C_4  \geq - ( A + E^2_1 ) w  \sum_i F^{i\bar i}  u_{i\bar i} - \frac{C_3 A^2}{w E^3_1} \sum_i F^{i\bar i}   .
\end{equation}

For $A \gg 1$ which is to be determined later, there are two cases in consideration:
(1) $w > A E^{-\frac{3}{2}}_1 \geq A > N$, where $N$ is the crucial constant in Lemma \ref{lemma-alternative};
(2) $w \leq A E^{-\frac{3}{2}}_1$ .

In the first case, by Lemma \ref{lemma-alternative},
\begin{equation}
    	(C_2 + C_3) w \sum_i F^{i\bar i}  + C_4 \geq   A w \theta \Big(\sum_i F^{i\bar i}  + 1\Big) .
\end{equation}
This gives a bound $w \leq 1$ at $p$ if we choose $A\theta > \max\{C_2 + C_3, C_4\} $. It contradicts the assumption $w \gg 1$.

In the second case,
\begin{equation}
\begin{aligned}
    w e^\phi \leq w e^\phi |_p \leq A E^{-\frac{3}{2}}_1 e^{ - A u + 1} |_p \leq A e^2 e^{- A  \inf_M u }
\end{aligned}
\end{equation}
and hence
\begin{equation}
\begin{aligned}
    w \leq A e^2 e^{ A u - E_1 - A  \inf_M u } \leq A e^2 e^{ A u - A  \inf_M u} \leq C e^{A (u - \inf_M u)} .
\end{aligned}
\end{equation}

\end{proof}

\bigskip

\section{The gradient estimate}
\label{gradient}
\setcounter{equation}{0}
\medskip

In this section, we provide a direct gradient estimate.
\begin{proposition}
 Let $u\in C^3(M)$ be a solution to equation \eqref{int-gcmate-equation}. Then there are uniform positive constants $C$ and $A$ such that
\begin{equation}
\label{gradient-1}
\sup_{M} |\nabla u|^2 \leq C e^{A (u - \inf_M u)},
\end{equation}
where $C$, $A$ depend on the given geometric quantities.
\end{proposition}
\begin{proof}
Let us consider the function $ \ln |\nabla u|^2 + \phi$ where $\phi$ is to be specified later.  Suppose that $\ln |\nabla u|^2 + \phi$ attains its maximal value at some point $p\in M$. We may assume that $|\nabla u|^2 \geq 1$. Choose a local chart near $p$ such that $g_{i\bar j} = \delta_{ij}$ and $X_{i\bar j}$ is diagonal at $p$. Therefore we have at the point $p$,
\begin{equation}
\label{gradient-derivative-1}
	\frac{\partial_i(|\nabla u|^2)}{|\nabla u|^2} + \partial_i \phi = 0 ,
\end{equation}
\begin{equation}
\label{gradient-derivative-2}
	\frac{\bar\partial_i(|\nabla u|^2)}{|\nabla u|^2} + \bar\partial_i\phi = 0 ,
\end{equation}
and
\begin{equation}
\label{gradient-derivative-3}
    \frac{\bar\partial_i\partial_i(|\nabla u|^2)}{|\nabla u|^2} - \frac{|\partial_i(|\nabla u|^2)|^2}{|\nabla u|^4} + \bar\partial_i\partial_i\phi \leq 0 .
\end{equation}

Direct calculation shows that
\begin{equation}
\label{gradient-derivative-4}
\begin{aligned}
    \partial_i(|\nabla u|^2) &= \sum_k (u_k u_{\bar k i} + u_{ki}u_{\bar k}) , 
\end{aligned}
\end{equation}
and
\begin{equation}
\label{gradient-derivative-5}
\begin{aligned}
    \bar\partial_i\partial_i(|\nabla u|^2) &= \sum_k   (u_{k\bar i}u_{\bar k i} + u_{ki}u_{\bar k\bar i} + u_{ki\bar i}u_{\bar k} + u_k u_{\bar k i\bar i})\\
    &= \sum_k |u_{k i}|^2 + \sum_k \Big|u_{\bar k i} - \sum_l T^k_{il} u_{\bar l}\Big|^2 + 2 \sum_k \mathfrak{Re}\Big\{X_{i\bar i k} u_{\bar k}\Big\}  \\
    &\quad - 2 \sum_k \mathfrak{Re}\Big\{\chi_{i\bar i k} u_{\bar k} \Big\} + \sum_k R_{i\bar i k\bar l} u_l u_{\bar k} - \sum_k \Big|\sum_l T^k_{il} u_{\bar l}\Big|^2 .
\end{aligned}
\end{equation}  
Substituting \eqref{gradient-derivative-5} into \eqref{gradient-derivative-3},
\begin{equation}
\label{gradient-derivative-6}
\begin{aligned}
	&\quad |\nabla u|^2 \sum_k |u_{k i}|^2  - |\partial_i(|\nabla u|^2)|^2 +|\nabla u|^4 \bar\partial_i\partial_i\phi\\
	&\leq - 2 |\nabla u|^2 \sum_k \mathfrak{Re}\Big\{X_{i\bar i k} u_{\bar k}\Big\} + 2 |\nabla u|^2 \sum_k \mathfrak{Re}\Big\{\chi_{i\bar i k} u_{\bar k} \Big\} \\
	&\quad - |\nabla u|^2 \sum_k R_{i\bar i k\bar l} u_l u_{\bar k} + |\nabla u|^2 \sum_k \Big|\sum_l T^k_{il} u_{\bar l}\Big|^2  \\
	&\leq - 2 |\nabla u|^2 \sum_k \mathfrak{Re}\Big\{X_{i\bar i k} u_{\bar k}\Big\} + C_1 |\nabla u|^4  .
\end{aligned}
\end{equation}


By Schwarz inequality,
\begin{equation}
\label{gradient-inq-2}
	\sum_i F^{i\bar i}  \Big|\sum_k u_{ki} u_{\bar k}\Big|^2 \leq |\nabla u|^2   \sum_i F^{i\bar i}  \sum_{k}|u_{ki}|^2  .
\end{equation}
By \eqref{gradient-derivative-4},
\begin{equation}
\label{gradient-inq-5}
\begin{aligned}
	\Big|\sum_k u_{ki} u_{\bar k}\Big|^2 
	&= \Big|\partial_i(|\nabla u|^2) - \sum_k u_k u_{\bar k i}\Big|^2 \\
	&= |\partial_i(|\nabla u|^2) |^2  +  \Big|\sum_k u_k u_{\bar k i}\Big|^2  \\
	&\quad - 2  X_{i \bar i} \mathfrak{Re} \Big\{ \partial_i(|\nabla u|^2) u_{\bar i} \Big\} + 2 \mathfrak{Re}\Big\{  \partial_i(|\nabla u|^2) \sum_k \chi_{k \bar i} u_{\bar k} \Big\} .
\end{aligned}
\end{equation}
Substituting \eqref{gradient-inq-5} and \eqref{gradient-derivative-1} into \eqref{gradient-inq-2},
\begin{equation}
\label{gradient-inq-6}
\begin{aligned}
	&\quad \sum_i F^{i\bar i}  \Big(|\nabla u|^2 \sum_{k}|u_{ki}|^2 - |\partial_i(|\nabla u|^2) |^2 \Big)\\
	&\geq 2 |\nabla u|^2  \sum_i F^{i\bar i} X_{i\bar i} \mathfrak{Re}\Big\{\partial_i\phi u_{\bar i}\Big\} - 2 |\nabla u|^2 \ \sum_i F^{i\bar i}  \mathfrak{Re}\Big\{\partial_i \phi \sum_k \chi_{k \bar i} u_{\bar k}\Big\} .
\end{aligned}
\end{equation}
Applying \eqref{gradient-derivative-6}, we obtain
\begin{equation}
\label{gradient-inq-8}
\begin{aligned}
	\frac{C_2}{ |\nabla u|}  + C_1  \sum_{i} F^{i\bar i}  &\geq \sum_i F^{i\bar i}  \bar\partial_i\partial_i\phi + \frac{2}{ |\nabla u|^2} \sum_i F^{i\bar i}  X_{i\bar i} \mathfrak{Re}\Big\{\partial_i\phi u_{\bar i}\Big\} \\
	&\quad - \frac{2}{ |\nabla u|^2} \sum_i F^{i\bar i}  \mathfrak{Re}\Big\{\partial_i \phi \sum_k \chi_{k \bar i} u_{\bar k}\Big\} .
\end{aligned}
\end{equation}

Set
\[
	\phi:= - A (u - \inf_M u) + \frac{1}{u - \inf_M u + 1} = - A(u - \inf_M u) + E_1 .
\]
Then
\begin{equation}
	\partial_i \phi = - (A + E^2_1) u_i ,
\end{equation}
and
\begin{equation}
		\bar\partial_i \partial_i \phi = - (A + E^2_1) u_{i\bar i} + 2 E^3_1 |u_i |^2 .
\end{equation}
So it follows from \eqref{gradient-inq-8},
\begin{equation}
\label{gradient-inq-9}
\begin{aligned}
	&\quad \frac{C_2}{ |\nabla u|}  + C_1 \sum_{i} F^{i\bar i}  + \frac{2 (A + E^2_1)}{ |\nabla u|^2}  \sum_i F^{i\bar i}  X_{i\bar i} |u_i|^2 \\
	&\quad - \frac{2 (A + E^2_1)}{|\nabla u|^2} \sum_i F^{i\bar i}  \mathfrak{Re} \Big\{ u_i \sum_k \chi_{k \bar i} u_{\bar k} \Big\} \\
	&\geq  - (A + E^2_1)  \sum_i F^{i\bar i}  u_{i\bar i} + 2 E^3_1 \sum_i F^{i\bar i}  |u_i |^2 .
\end{aligned}
\end{equation}

There are two cases: (1) $w \leq N$ and  (2) $w > N$.

In the first case, we know that there is $c_1 > 0$ such that $\frac{1}{c_1} \geq X_{i\bar i} \geq c_1$ for $i = 1, \cdots, n$.
So 
for some $\sigma_1 > 0$,
\begin{equation}
	\frac{1}{\sigma_1} \geq F^{i\bar i}  \geq \sigma_1 .
\end{equation}
Therefore
\begin{equation}
\label{gradient-inq-10}
\begin{aligned}
	\frac{C_2}{ |\nabla u|}  +\frac{n C_1}{\sigma_1} + \frac{2 n (A + 1) |\chi|}{\sigma_1} + \frac{( n + 2) (A + 1)}{ \psi }  \\
	\geq  \epsilon (A + E^2_1) \sum_i F^{i\bar i}  + 2 \sigma_1 E^3_1  |\nabla u |^2 ,
\end{aligned}
\end{equation}
and hence 
\begin{equation}
	|\nabla u|^2 \leq C_3 (A + 1) (u - \inf_M u + 1)^3 .
\end{equation}

In the second case,
\begin{equation}
\label{gradient-inq-12}
\begin{aligned}
	&\quad \frac{C_2}{ |\nabla u|}  + C_1 \sum_{i} F^{i\bar i}  + \frac{2 (A + E^2_1)}{ |\nabla u|^2}  \sum_i F^{i\bar i}  X_{i\bar i} |u_i|^2 \\
	&\quad - \frac{2 (A + E^2_1)}{|\nabla u|^2} \sum_i F^{i\bar i}  \mathfrak{Re} \Big\{ u_i \sum_k \chi_{k \bar i} u_{\bar k} \Big\} \\
	&\geq  (A + E^2_1) \theta\Big(  \sum_i F^{i\bar i}  + 1\Big) + 2 E^3_1 \sum_i F^{i\bar i}  |u_i |^2 .
\end{aligned}
\end{equation}
If $A \geq \frac{2(C_1 + C_2)}{ \theta}$,
\begin{equation}
\label{gradient-inq-13}
\begin{aligned}
	 \frac{2 (A + E^2_1)}{ |\nabla u|^2} \sum_iF^{i\bar i}  X_{i\bar i} |u_i|^2 - \frac{2 (A + E^2_1)}{|\nabla u|^2}  \sum_i F^{i\bar i} \mathfrak{Re} \Big\{ u_i \sum_k \chi_{k \bar i} u_{\bar k} \Big\} \\
	\geq  \frac{(A + E^2_1) \theta}{2} \Big( \sum_i F^{i\bar i}  + 1\Big) + 2 E^3_1 \sum_i F^{i\bar i}  |u_i |^2 .
\end{aligned}
\end{equation}
By Schwarz's inequality,
\begin{equation}
\label{gradient-inq-14}
\begin{aligned}
	&\quad \Bigg|\frac{2 (A + E^2_1)}{|\nabla u|^2} \sum_i F^{i\bar i}  \mathfrak{Re} \Big\{ u_i \sum_k \chi_{k \bar i} u_{\bar k} \Big\}\Bigg| \\
	&\leq \frac{(A + E^2_1) \theta}{4}  \sum_i F^{i\bar i}  + C_4  \frac{ (A + E^2_1)}{ |\nabla u|^2} \sum_i F^{i\bar i}  |u_i|^2 ,
\end{aligned}
\end{equation}
and hence
\begin{equation}
\label{gradient-inq-15}
\begin{aligned}
	&\quad \frac{2 (A + E^2_1)}{ |\nabla u|^2}  \sum_i F^{i\bar i} X_{i\bar i} |u_i|^2 +  C_4  \frac{ (A + E^2_1)}{ |\nabla u|^2}  \sum_i F^{i\bar i}  |u_i|^2  \\
	&\geq  \frac{(A + E^2_1) \theta}{4}  \sum_i F^{i\bar i}  + \frac{(A + E^2_1) \theta}{2}  + 2 E^3_1 \sum_i F^{i\bar i}  |u_i |^2 .
\end{aligned}
\end{equation}
Then we have either 
\begin{equation}
\label{gradient-inq-15-1}
	|\nabla u|^2 \leq  \frac{C_4 (A + 1)}{2 } (u - \inf_M u + 1)^3,
\end{equation}
or
\begin{equation}
\label{gradient-inq-15-2}
\frac{ \theta}{8}  \sum_i F^{i\bar i} \leq  \frac{1}{ |\nabla u|^2}  \sum_i F^{i\bar i}  X_{i\bar i} |u_i|^2 .
\end{equation}
In the later case, $S_1 \leq \frac{8 n}{\theta}$, which implies that $X^{i\bar i} < \frac{8 n}{\theta}$  for $i = 1, \cdots, n$.
Back to \eqref{gradient-inq-15}, we control the first term
\begin{equation}
\label{gradient-inq-15-4}
\begin{aligned}
	{\epsilon_1 (A + E^2_1)} \max_i \sum^n_{\a = 1} \frac{c_\a}{C^\a_n} S_{\a - 1;i} + \frac{(A + E^2_1)}{\epsilon_1  |\nabla u|^2} \sum_i F^{i\bar i}  |u_i|^2 +  C_4  \frac{ (A + E^2_1)}{ |\nabla u|^2}  \sum_i F^{i\bar i}  |u_i|^2  \\
	\geq  \frac{(A + E^2_1) \theta}{4} \sum_i F^{i\bar i}  + \frac{(A + E^2_1) \theta}{2} + 2 E^3_1  \sum_i F^{i\bar i} |u_i |^2 .
\end{aligned}
\end{equation}
Setting 
\[
\epsilon_1 = \frac{\theta}{2 \sum^n_{\alpha = 1} \frac{c_\alpha \alpha}{n} \Big(\frac{8 n}{\theta}\Big)^{\alpha - 1}} ,
\]
then
\begin{equation}
\label{gradient-inq-16}
	C_5 \frac{(A + E^2_1)}{ |\nabla u|^2} \sum_i F^{i\bar i}  |u_i|^2 \geq  \frac{(A + E^2_1) \theta}{4}  \sum_i F^{i\bar i}   + 2 E^3_1 \sum_i F^{i\bar i}  |u_i |^2 .
\end{equation}
So we obtain
\begin{equation}
\label{gradient-inq-17}
	|\nabla u|^2 \leq \frac{C_5 (A + 1)}{2} (u - \inf_M u + 1)^3 .
\end{equation}

Therefore, no matter in case (1), case (2) or $|\nabla u|^2 \leq 1$, we have at the point $p$,
\begin{equation}
\label{gradient-inq-18}
	|\nabla u|^2 \leq C (u - \inf_M u + 1)^3 ,
\end{equation}
if $A$ is chosen sufficiently large.
Without loss of generality, we may assume that $\ln |\nabla u|^2 + \phi \geq 0$, otherwise the proof is done. Thus
\begin{equation}
\label{gradient-test-1}
	\ln |\nabla u|^2 + 1 \geq A (u - \inf_M u) .
\end{equation}
Substituting \eqref{gradient-test-1} into \eqref{gradient-inq-18},
\begin{equation}
\label{gradient-test-2}
	|\nabla u|^2 \leq C_6 (\ln |\nabla u| + C_7)^3 ,
\end{equation}
which implies that $|\nabla u (p)|^2$ is bounded by a uniform constant $C_8$. Thus
\begin{equation}
	\ln |\nabla u|^2 + \phi \leq \ln |\nabla u|^2 + \phi \Big|_{x = p} \leq \ln C_8 + \phi (p) ,
\end{equation}
and hence
\begin{equation}
\begin{aligned}
	\ln |\nabla u|^2 
	&\leq A u  - \frac{1}{u - \inf_M u + 1} + \ln C_8 - A u(p)  + \frac{1}{u(p) - \inf_M u + 1} \\
\end{aligned}
\end{equation}

\end{proof}

\bigskip

\section{Solving the equations}
\label{solution}
\setcounter{equation}{0}
\medskip

In this section, we give proofs of the solvability results stated in introduction section.

\begin{proof}[Proof of Theorem~\ref{theorem-hermitian}]
First, we assume that $v$ is smooth. Define $\varphi$ by
\begin{equation}
\label{hermitian-continuity-method-flow}
	\chi^n_v = \varphi \sum^n_{\a = 1} c_\a \chi^{n - \a}_v \wedge \omega^\a.
\end{equation}

We use the continuity method and consider the family of equations
\begin{equation}
	\chi^n_{u_t} = \psi^t \varphi^{1 - t} e^{b_t} \sum^n_{\a = 1} c_\a \chi^{n - \a}_{u_t} \wedge \omega^\a, \qquad \text{for } t \in [0,1],
\end{equation}
where $\chi_{u_t} > 0$ and $b_t$ is a constant for each $t$. We consider the set
\begin{equation}
    \mathcal{T} := \{t'\in[0,1]\;|\; \exists \; u_t \in C^{2,\a}(M) \text{ and } b_t \text{ solving } \eqref{hermitian-continuity-method-flow} \text{ for } t\in[0,t']\}.
\end{equation}

As shown in \cite{Sun2013e}, the continuity method works if we can guarantee: (1) $0 \in \mathcal{T}$; (2) we have uniform $C^\infty$ estimates for all $u_t$. The first requirement is naturally met; for the second requirement, we only need to show that $\psi^t \varphi^{1 - t} e^{b_t} \leq \psi$ for all $t \in [0,1]$. At the maximum point of $u_t - v$, we have
\(
	\psi^t \varphi^{1 - t} e^{b_t} \leq \varphi,
\)
and thus $b_t \leq 0$. Therefore
\(
	\psi^t \varphi^{1 - t} e^{b_t} \leq \psi.
\)

Second, if $v$ is not smooth, we can smoothen $v$ and replace $\psi$ by $(1 + \delta) \psi$ for some very small $\delta > 0$.

\end{proof}

\begin{proof}[Proof of Theorem~\ref{theorem-kahler}]

Define $\varphi$ by
\begin{equation}
	\chi^n = \varphi \sum^n_{\a = 1} c_\a \chi^{n - \a} \wedge \omega^\a.
\end{equation}
Definitely,
\begin{equation}
	n \chi^{n - 1} > \varphi \sum^{n - 1}_{\a = 1} c_\a (n - \a) \chi^{n - \a - 1} \wedge \omega^\a.
\end{equation}
Therefore, we can find a $C^2$ function $h$ satisfying $h(x) \geq \max\{\varphi(x),\psi(x)\}$ for all $x \in M$ and
\begin{equation}
	n \chi^{n - 1} > h \sum^{n - 1}_{\a = 1} c_\a (n - \a) \chi^{n - \a - 1} \wedge \omega^\a.
\end{equation}

Since $h \geq \varphi$, by Theorem~\ref{theorem-hermitian}, there exists a  solution $u_0$ and a constant $b_0 \leq 0$ such that
\begin{equation}
	\chi^n_{u_0} = e^{b_0} h \sum^n_{\a = 1} c_\a \chi^{n - \a}_{u_0} \wedge \omega^\a .
\end{equation}

Now we apply continuity method again from $\chi_{u_0}$ and consider the family of equations
\begin{equation}
\label{kahler-continuity-method-flow-second}
	\chi^n_{u_t} = \psi^t h^{1 - t} e^{b_t} \sum^n_{\a = 1} c_\a \chi^{n - \a}_{u_t} \wedge \omega^\a, \qquad \text{for } t\in [0,1].
\end{equation}
According to the argument in \cite{Sun2013e}, we only need to show that $\psi^t h^{1 - t} e^{b_t} \leq h$ for all $t \in [0,1]$. Note that, we use the new condition~\eqref{kahler-continuity-method-flow-second} to obtain $C^\infty$ estimates.

Integrating equation flow~\eqref{kahler-continuity-method-flow-second},
\begin{equation}
\begin{aligned}
	\int_M \chi^n &= e^{b_t}  \sum^n_{\a = 1} \int_M \psi^t h^{1 - t} c_\a \chi^{n - \a}_{u_t} \wedge \omega^\a \\
	&\geq e^{b_t} c \sum^n_{\a = 1} c_\a \int \chi^{n - \a} \wedge \omega^\a ,
\end{aligned}
\end{equation}
which implies $b_t \leq 0$. Therefore,
\begin{equation}
	\psi^t h^{1 - t} e^{b_t} \leq \psi^t h^{1 - t} \leq h.
\end{equation}

\end{proof}

\bigskip

\noindent
{\bf Acknowledgements}\quad
The author is supported by China Postdoctoral Science Foundation (Grant No. 2015M571478) and National Natural Science Foundation of China (Grant No. 11501119).


\begin{thebibliography}{99}

\bibitem{Blocki09a}
Z. Blocki,
{\em A gradient estimate in the Calabi-Yau theorem},
Math. Ann. \textbf{344} (2009), no. 2, 317--327.

\bibitem{Chen00b}
X.-X. Chen,
{\em On the lower bound of the Mabuchi energy and its application},
Int. Math. Res. Notices \textbf{12}  (2000),  607--623.

\bibitem{Chen04}
X.-X. Chen,
{\em A new parabolic flow in K\"ahler manifolds},
Comm. Anal. Geom. {\bf 12} (2004), 837--852.

\bibitem{Cherrier87}
P. Cherrier,
{\em Equations de Monge-Amp\`ere sur les vari\'et\'es hermitiennes compactes},
Bull. Sci. Math. {\bf 111} (1987), 343--385.

\bibitem{CollinsSzekelyhidi2014a}
T. C. Collins and G. Sz\'ekelyhidi,
{\em Convergence of the $J$-flow on toric manifolds},
to appear in J. Differential Geom.

\bibitem{Donaldson99a}
S. K. Donaldson,
{\em Moment maps and diffeomorphisms},
Asian J. Math. {\bf 3} (1999), 1--16.


\bibitem{Evans82}
L. C. Evans,
{\em Classical solutions of fully nonlinear, convex, second order elliptic equations},
Comm. Pure Appl. Math. {\bf 35} (1982), 333--363.

\bibitem{FLM11}
H. Fang, M.-J. Lai and X.-N. Ma,
{\em On a class of fully nonlinear flows in K\"ahler geometry},
J. Reine Angew. Math. {\bf 653} (2011), 189--220.











\bibitem{Guan2014a}
B. Guan,
{\em Second order estimates and regularity for fully nonlinear elliptic equations on Riemannian manifolds},
Duke Math. J. {\bf 163} (2014), 1491--1524.





\bibitem{GSun12}
B.~ Guan and W.~Sun,
{\em On a class of fully nonlinear elliptic equations on Hermitian manifolds},
Calc. Var. PDE, DOI 10.1007/s00526-014-0810-1.




\bibitem{GLZ}
P.-F. Guan, Q. Li and X. Zhang, 
{\em A uniqueness theorem in K\"{a}hler geometry}, 
Math.Ann. {\bf 345} (2009), 377--393.


\bibitem{Krylov82}
N. V. Krylov,
{\em Boundedly nonhomogeneous elliptic and parabolic equations},
Izvestiya Ross. Akad. Nauk. SSSR {\bf 46} (1982), 487--523.


\bibitem{LejmiSzekelyhidi13}
M. Lejmi and G. Sz\'ekelyhidi,
{The J-flow and stability}.
Adv. Math. {\bf 274}  (2015), 404--431.






\bibitem{PhongSturm10}
D. H. Phong and J. Sturm,
{\em The Dirichlet problem for degenerate complex Monge-Amp\`ere equations},
Comm. Anal. Geom. {\bf 18} (2010), no. 1, 145--170.


\bibitem{PhongSturm12}
D. H. Phong and J. Sturm,
{\em On pointwise gradient estimates for the complex Monge-Amp\`ere equation},
in {\em Advances in Geometric Analysis}, 87¨C96, Adv. Lect. Math. \textbf{21}, International Press, 2012.

\bibitem{Pingali14}
V. Pingali,
{\em A fully nonlinear generalized Monge-Amp\`ere PDE on a torus},
Electron. J. Differ. Eq. {\bf 2014} (2014), no. 211. 1--8.

\bibitem{SW08}
J. Song and B. Weinkove,
{\em On the convergence and singularities of the J-flow with applications
to the Mabuchi energy},
Comm. Pure Appl. Math. {\bf 61} (2008), 210--229.




\bibitem{Sun2013e}
W. Sun,
{\em On a class of fully nonlinear elliptic equations on closed Hermitian manifolds},
J. Geom. Anal., DOI 10.1007/s12220-015-9634-2.

\bibitem{Sun2013p}
W. Sun,
{\em Parabolic complex Monge-Amp\`ere type equations on closed Hermtian manifolds},
Calc. Var. PDE {\bf 54}, no. 4 (2015), 3715--3733.

\bibitem{Sun2014e}
W. Sun,
{\em On a class of fully nonlinear elliptic equations on closed Hermitian manifolds II: $L^\infty$ estimate},
preprint, arXiv:1407.7630.

\bibitem{Sun2014u}
W. Sun,
{\em On uniform estimate of complex elliptic equations on closed Hermitian manifolds},
preprint, arXiv:1412.5001.

\bibitem{Szekelyhidi2014b}
G. Sz\'ekelyhidi,
{\em Fully non-linear elliptic equations on compact Hermitian manifolds},
preprint, arXiv:1501.02762.

\bibitem{TWWvY2014}
V. Tosatti, Y. Wang, B. Weinkove and X.-K. Yang,
{\em $C^{2,\alpha}$ estimates for nonlinear elliptic equations in complex and almost complex geometry},
{Calc. Var. PDE}
{\bf 54} (2015), no. 1, 431--453. 

\bibitem{TWv10a}
V. Tosatti and B. Weinkove,
{\em Estimates for the complex Monge-Amp\`ere equation on Hermitian and
balanced manifolds},
Asian J. Math. {\bf 14} (2010), 19--40.

\bibitem{TWv10b}
V. Tosatti and B. Weinkove,
{\em The complex Monge-Amp\`ere equation on compact Hermitian manifolds},
J. Amer. Math. Soc. {\bf 23} (2010), 1187--1195.

\bibitem{TWv11}
V. Tosatti and B. Weinkove,
{\em On the evolution of a Hermitian metric by its Chern-Ricci form},
J. Differential Geom. {\bf 99} (2015), no.1, 125--163.



\bibitem{TWv13b}
V. Tosatti and B. Weinkove,
{\em Hermitian metrics, (n-1,n-1) forms and Monge-Amp\`ere equations},
preprint, arXiv:1310.6326.




\bibitem{Weinkove04}
B. Weinkove,
{\em Convergence of the J-flow on K\"ahler surfaces},
Comm. Anal. Geom. {\bf 12} (2004), 949--965.

\bibitem{Weinkove06}
B. Weinkove,
{\em On the J-flow in higher dimensions and the lower boundedness
of the Mabuchi energy},
J. Differential Geom. {\bf 73} (2006), 351--358.

\bibitem{Yau78}
S.-T. Yau,
{\em On the Ricci curvature of a compact K\"ahler manifold and the complex
Monge-Amp\`ere equation. I.}
Comm. Pure Appl. Math. {\bf 31} (1978), 339--411.

\bibitem{Zhang2015}
D.-K. Zhang,
{\em Hessian equations on closed Hermitian manifolds},
preprint, arXiv:1501.03553.

\end{thebibliography}
\end{document}